\documentclass[12pt,leqno]{amsart}
\usepackage{amsmath, amssymb, amsthm, amsfonts}
\usepackage{bm}
\usepackage{mathtools}
\usepackage{enumerate}
\usepackage{enumitem}
\usepackage{mathrsfs}
\usepackage{hyperref}
\newtheorem{thm}{Theorem}[section]
\newtheorem{defi}{Definition}[section]

\newtheorem{cor}{Corollary}[section]

\newcommand{\be}{\begin{equation}}
\newcommand{\ee}{\end{equation}}
\numberwithin{equation}{section}
\newcommand{\bea}{\begin{eqnarray}}
\newcommand{\eea}{\end{eqnarray}}
\newcommand{\beb}{\begin{eqnarray*}}
\newcommand{\eeb}{\end{eqnarray*}}
\usepackage{amssymb,amsfonts,amsthm,setspace,indentfirst}
\usepackage[dvips]{graphics}
\usepackage{epsfig}
\begin{document}
\title{Rough convergence of sequences in a partial metric space}
\author{Amar Kumar banerjee$^{1}$ and Sukila khatun$^{2}$}
\address{$^{1}$Department of Mathematics, The University of Burdwan,
Golapbag, Burdwan-713104, West Bengal, India.} 
\address{$^{2}$Department of Mathematics, The University of Burdwan,
Golapbag, Burdwan-713104, West Bengal, India.}
\email{$^{1}$akbanerjee@math.buruniv.ac.in, akbanerjee1971@gmail.com}
\email{$^{2}$sukila610@gmail.com}
\begin{abstract}
In this paper we have studied the notion of rough convergence of sequences in a partial metric space. We have also investigated how far several relevant results on boundedness, rough limit sets etc. which are valid in a metric space are affected in a partial metric space. 
\end{abstract}
\subjclass[2020]{40A05, 40A99.}
\keywords{Rough convergence, partial metric spaces, rough limit sets.}
\maketitle
\section{\bf{Introduction}}
\noindent The idea of rough convergence of sequences was first introduced in a normed linear space by H. X. Phu\cite{PHU} in 2001 and then he discussed thoroughly on the properties of rough limit sets such as boundedness, closedness, convexity etc. In an infinite dimensional normed linear space, in another paper\cite{PHU1}, an extension work of \cite{PHU} were done. Later though many authors carried out the works of rough convergence \cite{AYTER2, PMROUGH1, PMROUGH2} in more generalized form. But the study of rough convergence in a metric space studied by S. Debnath and D. Rakhshit\cite{DR} and in a cone metric space studied by A. K. Banerjee and R. Mondal\cite{RMROUGH} are significant for our present works.\\
\indent In 1994, Matthews\cite{MATW} introduced the idea of partial metric spaces as a generalization of metric spaces using the notion of self-distance $d(x,x)$ which may not be zero where as in a metric space it is always zero. In our present work we discuss the idea rough convergence of sequences in a partial metric space. Also we have found out several properties of rough limit sets of sequences and other relevant properties in partial metric spaces.\\
We now recall some definitions and results which will be needed in the sequel.
\section{\bf{Preliminaries}}\label{preli}
\begin{defi}\cite{PHU}
Let $\{ x_{n} \}$ be a sequence in a normed linear space $(X, \left\| . \right\|)$,  and $r$ be a non negative real number. Then $\{ x_{n} \}$ is said to be rough convergent (or $r$-convergent) to $x$ of roughness degree $r$ if for any $\epsilon >0$, there exists a natural number $k$ such that $\left\|x_{n} - x \right\| < r +\epsilon$ for all $n \geq k$.
\end{defi}
\begin{defi}\cite{BMW}
A partial metric on a nonempty set $X$ is a function $p: X\times X \longrightarrow [0, \infty)$ such that for all $x,y,z \in X$:\\
$(p1)$ $0 \leq p(x,x) \leq p(x,y)$ (nonnegativity and small self-distances),\\
$(p2)$ $x=y \Longleftrightarrow p(x,x)=p(x,y)=p(y,y)$ (indistancy both implies equality),\\
$(p3)$  $p(x,y)= p(y,x)$ (symmetry),\\
$(p4)$  $p(x,y) \leq p(x,z) + p(z,y) - p(z,z)$ (triangularity).\\
A partial metric space is a pair $(X,p)$ such that $X$ is a nonempty set and $p$ is a partial metric on $X$.
\end{defi}
Examples of partial metric spaces and its properties have been thoroughly discussed in \cite{BMW}.\\
\begin{defi}\cite{BMW}
In a partial metric space  $(X,p)$, for $r>0$ and $x \in X$ we define the open and closed ball of radius $r$ and center $x$ respectively as follows  :\\ 
$B^{p}_{r}(x)=\{ y \in X : p(x,y)<p(x,x)+r  \}$ \\
$\overline{B^{p}_{r}}(x)=\{ y \in X : p(x,y) \leq p(x,x)+r  \}$ 
\end{defi}
\begin{defi}\cite{BMW}
In a partial metric space $(X,p)$, a subset $U$ of $X$ is said to be an open set if for every $x\in U$ there exists a $r>0$  such that $B^{p}_{r}(x) \subset U$.
\end{defi}
\begin{defi}\cite{BMW}
Let $(X,p)$ be a partial metric space. A subset $U$ of $X$ is said to be a bounded in $X$ if there exists a positive real number $M$ such that $sup$ $\{ p(x,y): x,y \in U\}< M$.
\end{defi}
\begin{defi}\cite{BMW}
Let $(X,p)$ be a partial metric space and $\{x_{n}\}$ be a sequence in $X$. Then $\{x_{n}\}$ is said to converge to $x \in X$ if and only if $lim_{n\to\infty}p(x_{n},x)=p(x,x)$; i.e. if for each $\epsilon > 0$ there exists $k \in \mathbb{N}$ such that 
$|p(x_{n},x)-p(x,x)|< \epsilon$ for all $n \geq k$.
\end{defi}
\begin{defi}\cite{BMW}
Let $(X,p)$ be a partial metric space and $\{x_{n}\}$ be a sequence in $X$. Then
 $\{x_{n}\}$ is called a Cauchy sequence if   $lim_{n,m\to\infty}p(x_{n},x_{m})$ exists and is finite; i.e. if there exists $a \geq 0$ such that for each $\epsilon > 0$ there exists positive integer $k$ such that $|p(x_{n},x_{m})-a|< \epsilon$, whenever $n,m>k$. 
\end{defi}
\begin{defi}\cite{BMW}
A partial metric space $(X,p)$ is said to be complete if every Cauchy sequence converges.
\end{defi}
\section{\bf{Rough convergence of sequences in a partial metric space}}
\indent Throughout $(X,p)$ or simply $X$ stands for an partial metric space, $\mathbb{R}$ for the set of all real numbers, $\mathbb{N}$ for the set of all natural numbers, sets are always subsets of $X$ unless otherwise stated.
\begin{defi}(cf \cite{PHU})
Let $(X,p)$ be a partial metric space. A sequence $\{ x_{n} \}$ in $X$ is said to be rough convergent(or $r$-convergent) to $x$ of roughness degree $r$ for some non-negative real number $r$ if for every  $\epsilon > 0$ there exists a natural number $k$ such that $| p(x_{n}, x)-p(x,x) |<r + \epsilon $ holds for all $n \geq k$.
\end{defi}
\indent Usually we denote it by $x_{n} \stackrel{r}{\longrightarrow} x$. It should be noted that when $r=0$ the rough convergence becomes the classical convergence of sequences in a partial metric space. If $\left\{x_{n} \right\}$ is $r$-convergent to $x$, then $x$ is said to be a $r$-limit point of $\left\{x_{n} \right\}$.  For some $r$ as defined above, the set of all $r$-limit points of a sequence $\left\{x_{n} \right\}$ is said to be the $r$-limit set of the sequence $\left\{x_{n} \right\}$ and we will denote it by $LIM^{r}x_{n}$. Therefore we can write $LIM^{r}x_{n}= \left\{x \in X : x_{n} \stackrel{r}{\longrightarrow} x\right\}$. We should note that $r$-limit point of a sequence $\left\{x_{n} \right\}$ may not be unique.\\
\newline \\
\indent The following example shows that a sequence which is rough convergent in a partial metric space may not be  convergent in that space.\\ 
\noindent \textbf{Example 3.1.} Let $X=\{0\}\cup\mathbb{N}$ and $p: X\times X \longrightarrow [0, \infty)$ be defined by

\begin{equation*}
    \ p(n,m)=\begin{cases}
    1+\frac{1}{n}+\frac{1}{m}, & \text{if $n,m \in \mathbb{N},n \neq m$}, \\
    1,                        & \text{if $n=m \in X$}, \\
1+\frac{1}{n},             & \text{if  $m=0,n \in \mathbb{N}$},  \\
1+\frac{1}{m},             & \text{if  $n=0,m \in \mathbb{N}$}.

    \end{cases}
\end{equation*}
Then $p(n,m)$ is a partial metric on $X$.\\
We consider the sequence $ \{z_{n}\}$= \{0,k,0,k,0,k,.......\} i.e.
\begin{equation*}
    \{z_{n}\}=\begin{cases}
0 & \text{ if $n$ is odd }\\
k & \text{ if $n$ is even}

               \end{cases}
\end{equation*}
Now, $lim_{n\to\infty}p(z_{2n+1},0)=1$ and
     $lim_{n\to\infty}p(z_{2n},0)=1+ \frac{1}{k}$.
So, $lim_{n\to\infty}p(z_n,0)$ dose not exist.
Again, $lim_{n\to\infty}p(z_{2n+1},k)=1+ \frac{1}{k}$ and
     $lim_{n\to\infty}p(z_{2n},k)=1$.\\
So, $lim_{n\to\infty}p(z_n,k)$ dose not exist.\\
Let $m \in \mathbb{N}$, $m \neq 0,k$. Then \\
 $lim_{n\to\infty}p(z_{2n+1}),m)=1+ \frac{1}{m}$ and
     $lim_{n\to\infty}p(z_{2n},m)=1+\frac{1}{k}+\frac{1}{m}$.\\
So, $lim_{n\to\infty}p(z_n,m)$ dose not exist.\\
Hence $\{z_{n}\}$ is not convergent to any number of $X$ in $(X,p)$. \\
Now, let $\epsilon>0$ be arbitary. Then
\begin{equation*}
    \ |p(z_{n},0)-p(0,0)|=\begin{cases}
|1-1|=0 < \frac{1}{k}+\epsilon & \text{ if $n$ is odd }\\
|(1+\frac{1}{k})-1|=\frac{1}{k} < \frac{1}{k}+\epsilon & \text{ if $n$ is even}.
               \end{cases}
\end{equation*}
And 
\begin{equation*}
    \ |p(z_{n},k)-p(k,k)|=\begin{cases}
|(1+\frac{1}{k})-1|=\frac{1}{k} < \frac{1}{k}+\epsilon & \text{ if $n$ is odd}\\
|1-1|=0 < \frac{1}{k}+\epsilon & \text{ if $n$ is even}.
               \end{cases}
\end{equation*}

Let $m\in X$ and $m \neq 0,k$. Then \\
\begin{equation*}
    \ |p(z_{n},m)-p(m,m)|=\begin{cases}
|1+\frac{1}{m}-1|=\frac{1}{m} < (\frac{1}{k}+\frac{1}{m})+\epsilon & \text{ if $n$ is odd}\\
|1+\frac{1}{k}+\frac{1}{m}-1|=\frac{1}{k}+\frac{1}{m} < (\frac{1}{k}+\frac{1}{m})  +\epsilon & \text{ if $n$ is even}.
               \end{cases}
\end{equation*}
Hence $\{z_{n}\}$ is rough convergent to $0$ and $k$ with roughness degree $r=\frac{1}{k}$ and rough converges to any number $m(\neq 0,k)$ of roughness degree $(\frac{1}{k}+\frac{1}{m})$.\\
\indent We will use the similar kind of concept as was in the case of a metric space for diameter of a set in a partial metric space. For a subset $A$ of $X$ the diameter of $A$ is defined by\\ $Diam(A)$ = $sup$ $\{ p(x,y) : x, y\in A\}$.\\ $A$ is said to be of infinite diameter in the case when supremum is not finite.\\
\indent It has been discussed by Phu\cite{PHU} in a normed linear space that the diameter of a $r$-limit set of a $r$-convergent sequence $\{ x_{n} \}$ of roughness degree $r$ in a normed linear space $X$ is not greater then $2r$. Here we have found out a similar kind of property in a partial metric space as follows.
\begin{thm}
Let $(X,p)$ be a partial metric space and $a$ be a positive real number such that $p(x,x)=a$ for all $x$ in $X$. Then the  diameter of a $r$- limit set of a $r$-convergent sequence $\{x_{n}\}$ in $X$ is not greater then $(2r+2a)$.
\end{thm}
\begin{proof}
We shall show that 
$diam(LIM^{r} x_{n})$=$sup$ \{$p(y,z): y,z \in LIM^{r} x_{n}\} \leq 2r+2a$.
If possible, let there exist elements $y,z \in LIM^{r} x_{n}$ such that $p(y,z)>2r+2a$.
So, for an $\epsilon \in (0,\frac{p(y,z)}{2}-r-a)$, there exist $k \in \mathbb{N}$ such that 
      $|p(x_{n},y)-p(y,y)|< r+\epsilon$ and
      $|p(x_{n},z)-p(z,z)|< r+\epsilon$, for all $n \geq k$.
       This implies that
       $\{p(x_{n},y)-p(y,y)\}< r+\epsilon$ and
       $\{p(x_{n},z)-p(z,z)\}<r+\epsilon$ for all $n \geq k$.
       Now for $n \geq k$, we can write
 \begin{equation*}
       \begin{split}
p(y,z) &\leq p(y,x_{n})+p(x_{n},z)-p(x_{n},x_{n})\\
       &=\{p(x_{n},y)-p(y,y)\}+\{p(x_{n},z)-p(z,z)\}-p(x_{n},x_{n})+p(y,y)+p(z,z)\\
       &< 2(r+\epsilon)-a+a+a\\
       &=2r+2\epsilon+a\\
       &<2r+ p(y,z)-2r-2a+a\\
       &=p(y,z)-a , \ \text{which is a contradiction}.\\   
       \end{split}
   \end{equation*}    
Hence there can not have elements $y,z \in LIM^{r} x_{n}$ such that $p(y,z)> 2r+2a$ holds.
Hence the result follows.

\end{proof}
\noindent \textbf{Remark :} We have taken an additional condition that $p(x,x)=a$ $\forall x$ in $X$ in the theorem 3.1. Truly there exists such kind of partial metric space. The partial metric spaces given in example 3.1 is such one.\\
We introduce the following definition of one sided convergence in a partial metric space.
\begin{defi}
A sequence $\{x_{n}\}$ is said to be convergent to $x$ from right (or left) in a partial metric space $(X,p)$ if for any $\epsilon>0$, $\exists k \in \mathbb{N}$ such that 
$| p(x_{n},x)-p(x,x) |< \epsilon $ holds for all $n \geq k$, when $p(x_{n}, x) \geq p(x,x)$ (or respectively when $ p(x_{n},x)\leq p(x,x)$).
\end{defi}
\begin{thm}
A sequence $\{x_{n}\}$ is convergent in $(X,p)$ iff it is both convergent from right and left.
\end{thm}
\begin{proof}
A sequence $\{x_{n}\}$ is convergent to $x$ in $(X,p)$, then $\{x_{n}\}$ is convergent to $x$ from right and left together holds trivialy.\\
Conversely, let $\{x_{n}\}$ be convergent to $x$ from right and from left together.
Then for any $\epsilon>0$ $\exists, k \in \mathbb{N}$ such that 
$| p(x_{n},x)-p(x,x) |< \epsilon $ holds for all $n \geq k$,
when $p(x_{n}, x) \geq p(x,x)$ 
and  $| p(x_{n},x)-p(x,x) |< \epsilon $ holds for all $n \geq k$,
when $p(x,x) \geq p(x_{n},x)$.
So, in any case $| p(x_{n},x)-p(x,x) |< \epsilon $ holds for all $n \geq k$.
This implies that $\{x_{n}\}$ is convergent to $x$.
\end{proof}
\begin{defi}
A sequence $\{x_{n}\}$ is said to be $r$-convergent to $x$ from right (or left) in a partial metric space $(X,p)$ if for any $\epsilon>0$, $\exists k \in \mathbb{N}$ such that 
$| p(x_{n},x)-p(x,x) |<r+ \epsilon $ holds for all $n \geq k$, when $p(x_{n}, x) \geq p(x,x)$ (or respectively when $ p(x_{n},x)\leq p(x,x)$).
\end{defi}
If $\{x_{n}\}$ is $r$-convergent to $x$ from right (or left) then $x$ is said to be right (or left) $r$-limit point of $\{x_{n}\}$.
\begin{thm}
A sequence $\{x_{n}\}$ is $r$-convergent of roughness degree $r$ in $(X,p)$ iff it is both $r$-convergent from right and left of same roughness degree $r$.
\end{thm}

Proof is similar to the proof of theorem 3.2.

We denote set of all right (or left) $r$-limit point of $\{x_{n}\}$ by $R-LIM^{r}x_{n}$ (or $L-LIM^{r}x_{n}$). We note that if $\{x_{n}\}$ is $r$-convergent then  $R-LIM^{r}x_{n} \subset LIM^{r}x_{n}$ and $L-LIM^{r}x_{n} \subset LIM^{r}x_{n}$.\\
\indent The following property is a modification of the result given in \cite{RMROUGH}.\\ 
\begin{thm}
If a sequence $\{x_{n}\}$ converges to $x$ in a partial metric space $(X,p)$, then $\{ y \in \overline{B^{p}_{r}}(x):p(x,x)=p(y,y)\} \subseteq R- LIM^{r}x_{n}$. 
\end{thm}
\begin{proof}
First suppose that the sequence $\{x_{n}\}$ converges to $x$ in a partial metric space $(X,p)$ and let $\epsilon >0$.
So, for $\epsilon >0$ there exist $k \in \mathbb{N}$ such that 
      $|p(x_{n},x)-p(x,x)|< \epsilon$, for all $ n \geq k$........(1).
 Let $y \in \overline {B^{p}_{r}}(x)$ such that $p(x,x)=p(y,y)$. Then $p(x,y) \leq p(x,x)+ r$ such that $p(y,y)=p(x,x)$........(2).
Now, by condition (p4) of the definition of partial metric space, we can write 
\begin{equation*}
    \begin{split}
        p(x_{n},y) &\leq p(x_{n},x)+p(x,y)-p(x,x)\\
             &\leq \{p(x_{n},x)-p(x,x)\}+p(x,y)\\
             &< \epsilon + \{p(x,x)+r\} \ \text{by} \ (1) \ \text{and} \ (2) \\
             &= p(x,x)+(r+\epsilon) \ ,\text{for} \ n\geq k.
    \end{split}
\end{equation*}
 Therefore, 
\begin{equation*}
    \begin{split}
p(x_{n},y)-p(y,y) & <p(x,x)-p(y,y)+(r+\epsilon)\\
                  & <(r+\epsilon) \ \text{, since} \ p(x,x)= p(y,y).
    \end{split}
\end{equation*}
So, when $p(x_{n},y) \geq p(y,y)$,  $|p(x_{n},y)-p(y,y)|=p(x_{n},y)-p(y,y) < (r+\epsilon)$, for all $n\geq k$.
Hence $y \in R-LIM^{r}x_{n}$.
\end{proof}
\noindent \textbf{Remark :} 
An open ball in a partial metric space $(X,p)$, by definition, contains the element $x$. The set of all open balls of a partial metric space $(X,p)$ forms a  basis for a topology on $X$, which is denoted by $\tau(p)$. 
It was shown in \cite{BMW} that $\tau(p)$ is $T_{0}$, but may not be Hausdorff, even may not be $T_{1}$. However the following is also true.

\begin{thm}
Every partial metric space $(X,p)$ is first countable.
\end{thm}
\begin{proof}
Consider $u=\{B^{p}_{\frac{1}{n}}(x): x \in X, n \in \mathbb{N}\}$.
Obviously $u \subset v$, where $v=\{B^{p}_\epsilon (x): x \in X ,\epsilon >0\}$, the basis of the topology $\tau(p)$.
We show that $u$ is a basis for  $\tau(p)$.
Let $A \in  \tau(p)$ and $x \in A $ be any element.
Since $v$ is a basis for $\tau(p)$, $\exists$ $\epsilon >0$ such that $x \in B^{p}_\epsilon \subset A$.
Choose $n \in \mathbb{N}$ so that $ \frac{1}{n} < \epsilon$.
Then 
$B^{p}_{\frac{1}{n}}(x) \subset  B^{p}_\epsilon (x)$ .
Clearly, $ x \in B^{p}_{\frac{1}{n}}(x)$.
So, $ x \in B^{p}_{\frac{1}{n}}(x) \subset  B^{p}_\epsilon (x)\subset A $.
So, $u$ form  a basis for  $\tau(p)$.
Clearly, $u$ is countable.
Thus $(X,p)$ is first countable.
\end{proof}

\begin{thm}
Let $\{x_{n}\}$ be a $r$-convergent sequence from right in $(X,p)$ and $\{y_{n}\}$ be a convergent sequence in $R-LIM^{r} x_{n}$ converging to $y$. Then $y$ must belongs to $R-LIM^{r} x_{n}$.
\end{thm}

\begin{proof}
Let $\epsilon>0$ pre assigned.
Since $\{y_{n}\}$ converges to $y$, for $\epsilon>0$ there exists $k_{1} \in \mathbb{N}$ such that
$|p(y_{n},y)-p(y,y)|< \frac{\epsilon}{2}$ for all $n \geq k_{1}$.
Now, let us choose $y_{m} \in R-LIM^{r} x_{n}$ with $m > k_{1}$.
Then for same $\epsilon>0$, there exists $k_{2} \in \mathbb{N}$ such that 
$|p(x_{n},y_{m})-p(y_{m},y_{m})|< r+ \frac{\epsilon}{2}$ for all $n \geq k_{2}$, when $p(x_{n},y_{m}) \geq p(y_{m},y_{m})$. So, $p(x_{n},y_{m})-p(y_{m},y_{m})< r+ \frac{\epsilon}{2}$ for all $n \geq k_{2}$......(1).
Since $m > k_{1}$, $|p(y_{m},y)-p(y,y)|< \frac{\epsilon}{2}$.......(2).
Also, for all $n \in \mathbb{N}$, we have 
$p(x_{n},y) \leq p(x_{n},y_{m})+p(y_{m},y)-p(y_{m},y_{m}).$\\
So, for all $n \geq k_{2}$ we have 
\begin{equation*}
    \begin{split}
    p(x_{n},y)-p(y,y) & \leq p(x_{n},y_{m})+p(y_{m},y)-p(y_{m},y_{m})-p(y,y) \\
    & \leq {p(x_{n},y_{m})-p(y_{m},y_{m})} + |p(y_{m},y)-p(y,y)| \\
    & <(r+ \frac{\epsilon}{2}) + \frac{\epsilon}{2}, \ \text{by} \ (1) \ \text{and} \ (2) \\
    & =r+\epsilon.
    \end{split}
\end{equation*}
So, when $p(x_{n},y) \geq p(y,y)$, $|p(x_{n},y)-p(y,y)|={p(x_{n},y)-p(y,y)}<r+\epsilon$ $\forall n \geq k_{2}$.\\
Hence $y \in R-LIM^{r} x_{n}$. 
\end{proof}
\begin{cor}
Let $\{x_{n}\}$ be a $r$-convergent sequence from right in a partial metric space $(X,p)$. Then $R-LIM^{r}x_{n}$ is a closed set for any degree of roughness $r \geq 0$.
\end{cor}
\begin{proof}
Since the partial metric space $(X,p)$ is first countable, by theorem 3.5, the results follows directly from theorem 3.6.
\end{proof}
\begin{thm}
Let $\{ x_{n}\}$ be a $r$-convergent sequence in $(X, p)$ and $\{y_{n}\}$ be a convergent sequence in $LIM^{r} x_{n}$ converging to $y$. Then $y$ must belongs to $R-LIM^{r} x_{n}$. 
\end{thm}

The proof is parallal to the proof of theorem 3.6 and so is omitted.

\indent A sequence $\{ x_{n}\}$ in a partial metric space $(X, p)$ is said to be bounded (cf.\cite{AKB}) if and only if there exists a $M(>0)\in \mathbb{R}$ such that $sup$ $\{p( x_{n}, x_{m})\}< M$.
The following theorem is a generalization of the classical properly of a sequence that a convergent sequence must be bounded. 
\begin{thm}
Every $r$-convergent sequence in a partial metric space $(X,p)$ is bounded. 
\end{thm}
\begin{proof}
Let $\{x_{n}\}$ be a sequence $r$-converges to $x$ in a partial metric space $(X,p)$ and let $p(x,x)=a$.
We will show that $\{x_{n}\}$ is bounded in $X$.
Let $\epsilon > 0$. So for $\epsilon > 0$, there exists a natural number $k$ such that 
$|p(x_{n},x)-p(x,x)|< r+ \epsilon$, for all $ n \geq k $.
Consider $L = max_{1\leq i,j \leq k} \{ p(x_{i}, x_{j}) \}$.\\  
Now, we consider three cases \\
\indent \textbf{Case(i):}
For $i \leq k$ and $j \geq k$, then 
$p(x_{j},x_{k}) \leq p(x_{j},x)+ p(x,x_{k})- p(x,x)$
$=p(x_{j},x)-p(x,x)+p(x_{k},x)-p(x,x)+p(x,x)$.
So, $|p(x_{j},x_{k})|\leq |p(x_{j},x)- p(x,x)|+ |p(x_{k},x)-p(x,x)|+|p(x,x)|$
$<2(r+\epsilon)+|a|$.......(1).
Now, $p(x_{i},x_{j}) \leq p(x_{i},x_{k})+ p(x_{k},x_{j})- p(x_{k},x_{k})$. 
So, $|p(x_{i},x_{j})| \leq |p(x_{i},x_{k})|+ |p(x_{k},x_{j})|+ |p(x_{k},x_{k})| $
 $< L+ 2(r+\epsilon)+ |a|+ L$
 $= 2(r+\epsilon+ L)+|a|$.\\
\indent \textbf{Case(ii):}
For $i \geq k$ and $j \leq k$, then interchanging the role of $i$ and $j$ in case(i), it follows that 
$p(x_{j},x_{j}) \leq  2(r+\epsilon+ L)+|a|$.\\
\indent \textbf{Case(iii):}
Now we consider the case for $i \geq k$ and $j \geq k$.
Since  $i \geq k$, so by (1)
$|p(x_{i},x_{k})| \leq  2(r+\epsilon)+|a|$
and since $j \geq k$,
$|p(x_{j},x_{k})| \leq  2(r+\epsilon)+|a|$.
Now, $p(x_{i},x_{j}) \leq p(x_{i},x_{k})+ p(x_{k},x_{j})- p(x_{k},x_{k})$. 
So, $|p(x_{i},x_{j})| \leq |p(x_{i},x_{k})|+ |p(x_{k},x_{j})|+ |p(x_{k},x_{k})|$
$< 4(r+\epsilon)+2|a|+L$.
If $M > max \{L,  2(r+\epsilon+L)+|a|, 2(r+\epsilon+L)+|a|, 4(r+\epsilon)+2|a|+L\}$,
then $p(x_{i},x_{j})< M $ for all $i,j \in \mathbb{N}$.
Therefore, $\{x_{n}\}$ is bounded in $X$ and hence the result follows.\\
\end{proof}
\indent We know that a bounded sequence in a metric space may not be convergent. But it has been studied in \cite{RMROUGH} that a bounded sequence in a cone metric space is rough convergent for some degree of roughness. This result in a partial metric space $(X,p)$, has been modified in the following way.
\begin {thm}
Let $\{x_n\}$ be a bounded sequence in a partial metric space $(X,p)$. Then $\{x_n\}$ is always $r$-convergent from right of roughness degree $r=4M$ and $r$-convergent from left of roughness degree $r=2M$.
\end{thm}
\begin{proof}
Let $\{ x_{n}\}$ be a bounded sequence in a partial  metric space $(X, p)$. So there exists a $M(>0)\in \mathbb{R}$ such that $sup$ $\{p(x_{n}, x_{m})\} < M$ for all $n ,m \in \mathbb{N}$......(1).
Let $k$ be a fixed natural number. Then for all $n \in \mathbb{N}$, we can write
$p(x_{n},x_{k}) \leq p(x_{n},x_{m})+ p(x_{m},x_{k})- p(x_{m},x_{m})$.
So, $p(x_{n},x_{k})-p(x_{k},x_{k}) \leq p(x_{n},x_{m})+ p(x_{m},x_{k})- p(x_{m},x_{m})-p(x_{k},x_{k})$.
Therefore, whenever $p(x_{n},x_{k}) \geq p(x_{k},x_{k})$, we have
\begin{equation*}
    \begin{split}
    |p(x_{n},x_{k})-p(x_{k},x_{k})| & =p(x_{n},x_{k})-p(x_{k},x_{k})\\
    & \leq p(x_{n},x_{m})+p(x_{m},x_{k})-p(x_{m},x_{m})-p(x_{k},x_{k}) \\
    & \leq |p(x_{n},x_{m})+p(x_{m},x_{k})-p(x_{m},x_{m})-p(x_{k},x_{k})| \\
    & \leq |p(x_{n},x_{m})|+ |p(x_{m},x_{k})|+ |p(x_{m},x_{m})|+|p(x_{k},x_{k})|  \\
    & < M+M+M+M \ \text{by} \ (1) \\
    & = 4M \\
    & < 4M+\epsilon \\
    & =r+ \epsilon \ \forall n, \ \text{where} \ r=4M. 
    \end{split}
\end{equation*}
Hence $\{ x_{n}\}$ rough converges to $x_{k}$ from right for every $k\in \mathbb{N}$ for degree of roughness $r$.
Again,

		$$p(x_{k},x_{k})  \leq p(x_{k},x_{n})+p(x_{n},x_{k})-p(x_{n},x_{n})$$

This implies that
	
	$$p(x_{k},x_{k})-p(x_{n},x_{k})  \leq p(x_{n},x_{k})-p(x_{n},x_{n})$$

Therefore, whenever $p(x_{n},x_{k})  \leq p(x_{k},x_{k})$,

\begin{equation*}
	\begin{split}
		|p(x_{n},x_{k})-p(x_{k},x_{k})| & =p(x_{k},x_{k})-p(x_{n},x_{k})\\
		& \leq p(x_{n},x_{k}) - p(x_{n},x_{n}) \\
		& \leq |p(x_{n},x_{k})-p(x_{n},x_{n})| \\
		& \leq |p(x_{n},x_{k})|+|p(x_{n},x_{n})| \\
		& < M+M \\
		& < 2M + \varepsilon
	\end{split}
\end{equation*}
Hence $\{x_n\}$ is $r$-convergent from left of roughness degree $r=2M$.

\end{proof}
The following result is a similar kind of result as discussed by Phu\cite{PHU} in a normed linear space.
\begin{thm}
Let $\{ x_{{n}_{k}}\}$ be a subsequence of $\{ x_{n}\}$ then   $LIM^{r} x_{n} \subseteq  LIM^{r} x_{{n}_{k}}$.
\end{thm}
\begin{proof}
Let $x \in  LIM^{r} x_{n}$ and $\epsilon (>0)$ be arbitary.
Then there exists a number $m \in \mathbb{N}$ such that 
$ |p(x_{n},x)-p(x,x)|< r+ \epsilon $, for all $ n \geq m $.
Let $n_{q}>m$ for some $q \in \mathbb{N}$. Then $n_{k}>m$ for all $k \geq q$.
Therefore $|p(x_{{n}_{k}},x)-p(x,x)|< r+ \epsilon$, for all $k \geq q$.
So $x \in LIM^{r} x_{{n}_{k}}$.
Hence the result follows.
\end{proof}
\begin {thm}
Let $\{x_{n}\}$ and $\{y_{n}\}$ be two sequences in a partial metric space $(X, p)$ such for every $\epsilon (>0)$ there exists $k \in \mathbb{N}$ such that $p(x_{n}, y_{n})\leq \epsilon$, when $n \geq k$.
 If $\{x_{n}\}$ is $r$-convergent to $x$ and for any $\epsilon >0$, $\exists$ $m \in \mathbb{N}$ such that $p(x_{n}, x_{n})\leq \epsilon$ when $n \geq m$, then $\{y_{n}\}$ is $r$-convergent to $x$ from right. 
 Conversely, if $\{y_{n}\}$ is $r$-convergent to $x$ and for any $\epsilon >0$, $\exists$ $p \in \mathbb{N}$ such that $p(y_{n}, y_{n})\leq \epsilon$, when $n \geq p$, then $\{x_{n}\}$ is $r$-convergent to $x$ from right.  
\end{thm}
\begin{proof}
Let $\{x_{n}\}$ be  $r$-convergent to $x$ and let $\epsilon (>0)$ be preassigned.
So, for $\epsilon (>0)$ there exists $k_{1} \in \mathbb{N}$ such that 
$|p(x_{n},x) - p(x,x)|< r+ \frac{\epsilon}{3}$ for all $n \geq k_{1}$.
Again by the conditions, there exists $k_{2},k_{3} \in \mathbb{N}$ such that $p(x_{n}, y_{n})\leq \frac{\epsilon}{3}$ for all $n \geq k_{2}$ and $p(x_{n}, x_{n})\leq \frac{\epsilon}{3}$ for all $n \geq k_{3}$.
Let $k=max \{k_{1},k_{2},k_{3}\}$.Then for all $n \geq k$, we have
$|p(x_{n},x) - p(x,x)|< r+ \frac{\epsilon}{3}$,
$p(x_{n}, y_{n})\leq \frac{\epsilon}{3}$ and
$p(x_{n}, x_{n})\leq \frac{\epsilon}{3}$.......(1).
Again, for all $n \in \mathbb{N}$, we can write 
$p(y_{n},x) \leq p(y_{n},x_{n})+p(x_{n},x)-p(x_{n},x_{n})$.
Now, $p(y_{n},x) - p(x,x) \leq p(y_{n},x_{n}) + p(x_{n},x) - p(x_{n},x_{n}) - p(x,x)$.\\
So, whenever $p(y_{n},x) \geq p(x,x)$, then
\begin{equation*}
    \begin{split}
    |p(y_{n},x) - p(x,x)|& =p(y_{n},x) - p(x,x)\\
    & \leq p(y_{n},x_{n}) + p(x_{n},x) - p(x_{n},x_{n}) - p(x,x)\\
    & \leq |p(y_{n},x_{n}) + p(x_{n},x) - p(x_{n},x_{n}) - p(x,x)|\\
    &\leq |p(x_{n},y_{n})| + |p(x_{n},x) - p(x,x)| + |p(x_{n},x_{n})|\\
    & < \frac{\epsilon}{3} + (r+ \frac{\epsilon}{3})+\frac{\epsilon}{3}, \ \text{by} \ (1)\\
    &  = r + \epsilon \ \text{for all} \ n \geq k.
    \end{split}
\end{equation*}
Therefore, $\{y_{n}\}$ is $r$-convergent to $x$ from right.\\
Converse part is similar.
\end{proof}
\begin{thm}
Let $\{x_{n}\}$ and $\{y_{n}\}$ be two sequences in $(X, p)$ such that for every $\epsilon >0$, $\exists$ $k \in \mathbb{N}$ such that $p(x_{n}, y_{n})\leq \epsilon$, when $n \geq k$.
If $\{x_{n}\}$ is $r$-convergent to $x$ and if $\exists$ a positive number $c$ such that $ p(x_{n}, x_{n}) \leq c$ $\forall n$ (i.e. self-distance of the sequence is less or equal to $c$) then  $\{y_{n}\}$ is $(r + c) $-convergent to $x$ from right. 
Conversely, if $\{y_{n}\}$ is $r$-convergent to $x$ and if $\exists$ a positive number $d$ such that $p(y_{n}, y_{n})\leq$ d $\forall n$, then $\{x_{n}\}$ is $(r + d)$-convergent to $x$ from right.  
\end{thm}

The proof is parallel to the proof of the above theorem and so is omitted. 

\begin {thm}
Let $(X, p)$ be a partial metric space. If  $\{x_{n}\}$ and  $\{y_{n}\}$ be two sequences rough convergent to $x$ and $y$ respectively of roughness degree $r$ in $X$, then the real sequence $\{p(x_{n},y_{n})\}$ is rough convergent to $p(x,y)$ from right of roughness degree $2r$. 
\end{thm}
\begin{proof}
 Let $\{x_{n}\}$ be $r$-convergent to $x$ and $\{y_{n}\}$ be $r$-convergent to $y$.
 Then for any preassigned $\epsilon (>0)$ there exist $ k_{1} , k_{2} \in \mathbb{N}$ such that 
 $|p(x_{n},x)-p(x,x)|< r+ \frac{\epsilon}{2}$ for all $n \geq k_{1}$ and
 $|p(y_{n},y) - p(y,y)|< r+ \frac{\epsilon}{2}$ for all $n \geq k_{2}$.
 Let $k=max \{k_{1},k_{2}\}$. Then for all $n \geq k$, we have
 $|p(x_{n},x) - p(x,x)|< r+ \frac{\epsilon}{2}$ 
 and $|p(y_{n},y) - p(y,y)|< r+ \frac{\epsilon}{2}$......(1).
 Again, for all $n \in \mathbb{N}$ we can write \\
 $ p(x_{n},y_{n}) \leq  p(x_{n},x) + p(x,y) +  p(y,y_{n}) - p(x,x) - p(y,y)$.
 Now,
 \begin{equation*}
     \begin{split}
    p(x_{n},y_{n}) - p(x,y) & \leq  p(x_{n},x) - p(x,x) + p(y_{n},y) - p(y,y) \\  
    & < (r + \frac{\epsilon}{2}) + (r + \frac{\epsilon}{2}) \ \text{by} \ (1) \\
    & = 2r + \epsilon \ \text{for all} \ n \geq k.
     \end{split}
 \end{equation*}
So, whenever $p(x_{n},y_{n}) \geq p(x,y)$,
$| p(x_{n},y_{n}) - p(x,y)|< 2r + \epsilon$ for all $n \geq k$.
Hence the result.
\end{proof}
\vspace{0.2in}
\indent Let $\{x_{n}\}$ be a sequence in a partial metric space $(X, p)$. Then a point $c \in X$ is said to be a cluster point of $\{x_{n}\}$ if for every $\epsilon (>0)$ and every natural number $k$, there exists a natural number $k_{1}$ such that $k_{1}>k$ with  $|p(x_{{k}_{1}}, c) - p(c,c)|< \epsilon$. The idea of a closed ball has been discussed previously.\\
\indent The following result has been discussed in a cone metric space\cite{RMROUGH} here we have verified whether the same holds in a partial metric space. Indeed it gives a range of rough limit sets in terms of cluster points.
\begin {thm}
Let $(X, p)$ be partial metric space and $a$ be a real constant such that $p(x,x)=a$, $\forall x \in X$. Let $\{x_{n}\}$ be a sequence in $(X, p)$. If $c$ is a cluster point of $\{x_{n}\}$, then $LIM^{r}x_{n}  \subset \overline{B^{p}_{r}}(c)$ for some $r >0$.    
\end{thm}
\begin{proof}
Let $y \in LIM^{r}x_{n}$ but $y \notin \overline{B^{p}_{r}}(c)=\{y \in X: p(c,y) \leq p(c,c) + r\} =\{y \in X: p(c,y) \leq a + r\}$.
So, $a + r < p(c,y)$.
Let $\epsilon^{'}= p(c,y) - a -r$ and so $p(c,y)= \epsilon^{'} + a + r$, where $\epsilon^{'} > 0$.
Let $\epsilon =\frac{\epsilon^{'}}{2}$ and so we can write 
$p(c,y)= 2\epsilon + a + r$.
Then $B_{r+\epsilon}(y) \cap B_{\epsilon}(c) = \phi$.
For, otherwise, if $ q \in B_{r+\epsilon}(y) \cap B_{\epsilon}(c)$,
then $p(q,y) < p(y,y) + r + \epsilon = a + r + \epsilon$  and $p(q,c) < p(c,c) + \epsilon = a + \epsilon$.
Again, $p(c,y) \leq p(c,q) + p(q,y) - p(q,q) < \{a+\epsilon\} + \{a+r+\epsilon\} -a = a+r+2\epsilon = p(c,y)$, which is a contradiction.
Therefore, $B_{r+\epsilon}(y) \cap B_{\epsilon}(c) = \phi$.
But since $y \in LIM^{r}x_{n}$, for $\epsilon > 0$ there exists a $k_{0} \in \mathbb{N}$ such that $|p(x_{n},y)-p(y,y)| < r+\epsilon$ for all $n \geq k_{0}$.
Again, since $c$ is a cluster point of $\{x_{n}\}$, for $\epsilon > 0$ and for $k_{0} \in \mathbb{N}$, there exists a $k_{1} \in \mathbb{N}$ with $k_{1}>k_{0}$ such that $|p(x_{k_{1}},c)-p(c,c)| < \epsilon$.
So, $p(x_{k_{1}},c)-p(c,c) < \epsilon$. This implies that $p(x_{k_{1}},c) < p(c,c) + \epsilon$. Hence $ x_{k_{1}} \in B_{\epsilon}(c)$.
Also, $|p(x_{k_{1}},y)-p(y,y)| < r+\epsilon$.
So, $p(x_{k_{1}},y)-p(y,y) < r+\epsilon$. This implies that $p(x_{k_{1}},y) < p(y,y) + r + \epsilon$. Hence $ x_{k_{1}} \in B_{r+\epsilon}(y)$.
So, $ x_{k_{1}} \in B_{r+\epsilon}(y) \cap B_{\epsilon}(c)$, which is a contradiction.
Hence  $y \in \overline{B^{p}_{r}}(c)$.
\end{proof}
\noindent \textbf{Acknowledgement:}
The second author is thankful to The University of Burdwan for the grant of Junior Research Fellowship (State Funded) during the preparation of this paper.

\end{document}